\documentclass[a4paper,12pt]{amsart}
\usepackage[utf8]{inputenc}
\usepackage[english]{babel}
\usepackage{amsaddr}
\usepackage{calrsfs}
\usepackage{amssymb}

\usepackage{float}
\usepackage{graphicx,color}

\makeatletter
\@namedef{subjclassname@2020}{%
	\textup{2020} Mathematics Subject Classification}
\makeatother

\numberwithin{equation}{section}

\theoremstyle{definition}

\theoremstyle{plain}
\newtheorem{thm}{Theorem}[section]
\newtheorem{pro}{Proposition}[section]

\theoremstyle{definition}
\newtheorem{rem}{Remark}[section]

\newcommand{\E}{\mathbb{E}}

\newcommand{\F}{\mathcal{F}}

\newcommand{\e}{\mathrm{e}}
\newcommand{\la}{\langle}
\newcommand{\ra}{\rangle}
\renewcommand{\H}{\mathcal{H}}
\newcommand{\1}{\mathbf{1}}
\renewcommand{\d}{\mathrm{d}}

\begin{document}
\title{Transfer principle for fractional Ornstein--Uhlenbeck processes}

\date{\today}

\author[Sottinen]{Tommi Sottinen}
\address{School of Technology and Innovations, University of Vaasa, P.O. Box 700, FIN-65101 Vaasa, FINLAND}
\email{tommi.sottinen@uwasa.fi}

\author[Viitasaari]{Lauri Viitasaari}
\address{Department of Mathematics, Uppsala University, Box 480, 751 06 Uppsala, SWEDEN}
\email{lauri.viitasaari@math.uu.se}


\begin{abstract}
We prove the transfer principle for fractional Ornstein--Uhlenbeck processes, i.e., we construct a Brownian motion that has the same filtration as the fractional Ornstein--Uhlenbeck process and then represent the fractional Ornstein--Uhlenbeck process by using the constructed Brownian motion.  As applications of the transfer principle we consider prediction.
\end{abstract}

\keywords{fractional Brownian motion; fractional Ornstein--Uhlenbeck process; Langevin equation; Transfer principle; Prediction}

\subjclass[2020]{60G15, 60G18, 60G22, 60G25}

\maketitle


\section{Introduction}\label{sect:introduction}
The fractional Ornstein--Uhlenbeck process is defined by the Langevin equation
\begin{equation}\label{eq:Langevin}
\d U^{\theta,\sigma,H}_t
=
-\theta U^{\theta,\sigma,H}_t\d t + \sigma \d B^H_t, \quad U_0=0,
\end{equation}
with parameters $\theta,\sigma>0$ and $H\in(0,1)$.
The fractional Ornstein--Uhlenbeck process was first studied by Cheridito et al. \cite{Cheridito-Kawaguchi-Maejima-2003} and it has recently been studied extensively especially in connection with parameter estimation, see e.g. \cite{Sottinen-Viitasaari-2018b} and references therein. The motivation stems from the fact that Langevin equations characterise stationary processes (with suitably chosen initial condition $U_0$) \cite{Viitasaari-2016a,voutilainen-2021}, and hence provide a rich and tempting class for modelling. In particular, the fractional Ornstein--Uhlenbeck process defined as the solution \eqref{eq:Langevin} provides a (stationary) Gaussian model for which one can fine tune path regularity properties or incorporate short/long memory by altering the Hurst parameter $H$. 

The transfer principle for a given Gaussian process $X$ allows to construct an associated Brownian motion $W$ such that $X$ and $W$ can be recovered from each others. This property allows to transfer computations from more complicated process $X$ on to the well-understood ''Brownian side'' and then back, hence the term transfer principle. This has turned out to be powerful tool in various applications such as equivalence in law or in prediction, see e.g. \cite{Dufitinema-Shokrollahi-Sottinen-Viitasaari-2021,Sottinen-Viitasaari-2017b} for such applications in some fractional Gaussian models.

In this article we provide the transfer principle for the fractional Ornstein-Uhlenbeck process defined by \eqref{eq:Langevin}. In particular, we provide explicit formulas for the transfer operators allowing to change from $U^{\theta,\sigma,H}_t$ into the generating Brownian motion $W$ and back. We begin by recalling basic facts on abstract Wiener integration in Section \ref{sec:Wiener} and on fractional Brownian motion and its transfer principle on Section \ref{sec:fbm-transfer}. Our main result, the transfer principle for the fractional Ornstein-Uhlenbeck process, is proved in Section \ref{sec:fou-transfer}. We end the paper with a simple application to prediction in Section \ref{sec:prediction}.

\section{Abstract Wiener integration}
\label{sec:Wiener}
For a centered Gaussian process $X=(X_t)_{t\ge 0}$ the abstract Wiener integral
$$
\int_0^T f(t)\, \d X_s
$$
is defined as follows. Let 
$$
R(t,s) = \E[X_tX_s]
$$
be the covariance of $X$.  Define the integrand space $\Lambda([0,T])$ to be the closure of indicator functions $\1_t = \1_{[0,t)}$, $t\in[0,T]$, under the inner product generated by the relation
$$
{\la \1_t, \1_s\ra}_{\Lambda([0,T])} = R(t,s).
$$
Then $\Lambda([0,T])$ is a Hilbert space and the mapping $\mathcal{I}_T:\1_t\mapsto X_t$ extends to an isometry between $\Lambda([0,T])$ and the Gaussian subspace $\H_1([0,T])$ of $L^2(\Omega)$ generated by the random variables $X_t$, $t\in[0,T]$. For $f\in\Lambda([0,T])$, we set
$$
\int_0^T f(t)\, \d X_t = \mathcal{I}_T f.
$$

Note that $\Lambda([0,T])$ need not be a function space (cf. Pipiras and Taqqu \cite{Pipiras-Taqqu-2001}).  If however it contains smooth functions and $f\in\Lambda([0,T])$ is a smooth function, then the abstract Wiener integral can be understood in the integration-by-parts sense as
$$
\int_0^T f(t)\, \d X_t 
=
f(T)X_T - f(0)X_0 - \int_0^T X_t f'(t)\, \d t.
$$ 


\section{Fractional Brownian motion and its transfer principle}
\label{sec:fbm-transfer}
We shortly define the fractional Brownian motion and present its transfer principle, see equations \eqref{eq:W}--\eqref{eq:B} below.  For proofs and more details we refer to Mishura \cite{Mishura-2008}.

The fractional Brownian motion $B^H=(B^H_t)_{t\ge0}$ with Hurst index $H\in(0,1)$ is the centered Gaussian process having covariance function
$$
R_H(t,s) = \E[B_t^H B_s^H]
=
\frac12\left[t^{2H} + s^{2H} - |t-s|^{2H}\right].
$$
It is easy to see that the fractional Brownian motion can also be characterized as being the (upto a multiplicative constant) unique Gaussian process that has stationary increments and self-similarity of order $H$:
\begin{itemize}
\item $B^H_t-B^H_s \stackrel{d}{=} B^H_{t-s}$, \\
\item $B^H_{at} \stackrel{d}{=} a^{-H} B_t$.
\end{itemize}
Also, it is easy to see (cf. \cite{Azmoodeh-Sottinen-Viitasaari-Yazigi-2014}) that the fractional Brownian motion is H\"older continuous with any index $H-\varepsilon$ for any $\varepsilon>0$, but not H\"older continuous with index $H$.

Recall Riemann--Liouville fractional integral ($I_{T-}^\alpha$) and differential ($I_{T-}^{-\alpha}$) operators on an interval $[0,T]$:
\begin{eqnarray*}
I^\alpha_{T-}[f](t) &=& \frac{1}{\Gamma(\alpha)}\int_t^T f(s)(s-t)^{\alpha-1}\, \d u, \\
I^{-\alpha}_{T-}[f](t) &=& -\frac{1}{\Gamma(1-\alpha)}\frac{\d }{\d u}\int_t^T f(s)(s-t)^{-\alpha}\, \d s.
\end{eqnarray*}
Here $\Gamma$ is the gamma function
$$
\Gamma(z) = \int_0^\infty t^{z-1}\e^{-t}\, \d t.
$$
Let $c_H$ be te normalizing constant
$$
c_H = \sqrt{\frac{\Gamma(H-\frac12)(H-\frac12)2H}{\mathrm{B}(H-1\frac12,2-2H)}},
$$ 
where $\mathrm{B}$ is the beta function
$$
\mathrm{B}(z_1,z_2) = \frac{\Gamma(z_1)\Gamma(z_2)}{\Gamma(z_1+z_2)}.
$$
Denote
\begin{eqnarray}\label{eq:K}
K_H(t,s) &=& 
c_H t^{\frac12-H}I_{T-}^{H-\frac12}[(\cdot)^{H-\frac12}\1_t](s)\\
\label{eq:Kinv}
K^{-1}_H(t,s) &=& 
\frac{1}{c_H}t^{\frac12-H}I_{T-}^{\frac12-H}[(\cdot)^{H-\frac12}\1_t](s).
\end{eqnarray}
Then the abstract Wiener integral \eqref{eq:W} below defines a Brownian motion.  Moreover the fractional Brownian motion can be recovered from the Brownian motion \eqref{eq:W} by using the Wiener integral \eqref{eq:B} 
\begin{eqnarray}\label{eq:W}
W_t &=& \int_0^t K^{-1}_H(t,s)\, \d B^H_s, \\
\label{eq:B}
B^H_t &=& \int_0^t K_H(t,s)\, \d W_s.
\end{eqnarray}

Let $\Lambda_H([0,T])$ be the space of Wiener integrands for the fractional Brownian motion over the time interval $[0,T]$.  See Pipiras and Taqqu \cite{Pipiras-Taqqu-2001} for characterization of the spaces $\Lambda_H([0,T])$. For simple functions $f = \sum_{k=1}^n \alpha_k I_{[s_k,t_k]} \in \Lambda_H([0,T])$, we define the linear operator $K^*_H$ as
\begin{eqnarray}
K^*_H[g](t) &=& g(t)K_H(T,t) + \int_t^T [g(s)-g(t)] K_H(\d s, t) \nonumber\\
&=& c_H t^{\frac12-H}I_{T-}^{H-\frac12}[(\cdot)^{H-\frac12}g](t)\label{eq:K-star}.
\end{eqnarray}
By linearity, $K^*$ extends into an isometric operator from $\Lambda_H([0,T])$ into $L^2([0,T])$ and its inverse $(K^*_H)^{-1}\colon L^2([0,T]) \mapsto \Lambda_H([0,T])$, for nice enough function $g\in L^2([0,T])$, is given by 
\begin{eqnarray}
(K^*_H)^{-1}[g](t) &=& g(t)K^{-1}_H(T,t) + \int_t^T [g(s)-g(t)] K^{-1}_H(\d s, t) \nonumber \\
&=& \frac{1}{c_H} t^{\frac12-H}I_{T-}^{\frac12-H}[(\cdot)^{H-\frac12}g](t) \label{eq:K-star-inv}.
\end{eqnarray}
\begin{rem} 
It is evident from representations
$$
K^*_H[g](t) = c_H t^{\frac12-H}I_{T-}^{H-\frac12}[(\cdot)^{H-\frac12}g](t)
$$
and 
$$
(K^*_H)^{-1}[g](t) = \frac{1}{c_H} t^{\frac12-H}I_{T-}^{\frac12-H}[(\cdot)^{H-\frac12}g](t)
$$
why $\Lambda_H([0,T])$ may contain distributions for $H>\frac12$ while it is a subspace of $L^2([0,T])$ for $H<\frac12$. Indeed,  $I_{T-}^{\frac12-H}$ for $H<\frac12$ is a smoothing operator making $L^2([0,T])$ functions more regular. In contrary, for $H>\frac12$ the operator $I_{T-}^{H-\frac12}$ is a smoothing operator that can lift distributions into $L^2([0,T])$ functions.  
\end{rem}
The transfer principle \eqref{eq:W}--\eqref{eq:B} extends immediately to Wiener integration. Indeed, for any $f\in L^2([0,T])$ and $g\in\Lambda_H([0,T])$ we have
\begin{eqnarray*}\label{eq:intW}
\int_0^T f(t)\, \d W_t &=& \int_0^T (K^*_H)^{-1}f(t)\, \d B^H_t, \\
\label{eq:intB}
\int_0^T g(t)\, \d B^H_t &=& \int_0^T K^*_H g(t)\, \d W_t.
\end{eqnarray*}


\section{Fractional Orstein--Uhlenbeck processes and their transfer principle}
\label{sec:fou-transfer}
Consider now the Langevin equation \eqref{eq:Langevin}, i.e. 
\begin{equation*}
\d U^{\theta,\sigma,H}_t
=
-\theta U^{\theta,\sigma,H}_t\d t + \sigma \d B^H_t, \quad U_0=0,
\end{equation*}
with parameters $\theta,\sigma>0$ and $H\in(0,1)$. The equation can be solved and understood by using integration-by-parts.  Indeed, we obtain
$$
U^{\theta,\sigma,H}_t = \e^{-\theta t}\int_0^t \e^{\theta s} \sigma\, \d B^H_s.
$$
The stationary solution $V^{H,\sigma,H}$ to the Langevin equation is 
$$
V_t^{\theta,H,\sigma} = \e^{-\theta t}\int_{-\infty}^t \e^{\theta s}\sigma\, \d B_t^H. 
$$ 
Thus we have the simple connection
$$
V_t^{\theta,\sigma, H}
=
\e^{-\theta t}V_0^{\theta,\sigma,H} + U_t^{\theta,\sigma,H}.
$$

Denote
\begin{eqnarray}\label{eq:L}
L_{\theta,\sigma,H}(t,s) &=& 
\sigma K_H(t,s)-\theta\sigma\e^{-\theta t}\int_s^t K_H(u,s)\e^{\theta u}\, \d u, \\
\label{eq:Linv}
L^{-1}_{\theta,\sigma,H}(t,s) &=& 
\frac{1}{\sigma}K^{-1}_H(t,s) 
+\frac{\theta}{\sigma}(t-s).
\end{eqnarray}
We are now ready to present our transfer principle for $U^{\theta,\sigma,H}$.

\begin{thm}[Transfer principle]\label{thm:tp}
Let $L_{\theta,\sigma,H}$ be given by \eqref{eq:L} and let $L^{-1}_{\theta,\sigma,H}$ be given by \eqref{eq:Linv}. Then 
\begin{equation}\label{eq:WfOU}
W_t =\int_0^t L^{-1}_{\theta,\sigma,H}(t,s)\, \d U^{\theta,\sigma,H}_s	
\end{equation}
is the Brownian motion that generates $B^H$ as in \eqref{eq:B}, and $U^{\theta,\sigma,H}$ is recovered from it by 
\begin{equation}\label{eq:UfOU}
U^{\theta,\sigma,H}_t = \int_0^t L_{\theta,\sigma,H}(t,s)\, \d W_s.	
\end{equation}
In particular, $W$ and $U^{\theta,\sigma,H}$ generate the same filtration.
\end{thm}

\begin{proof}
Let us first consider the representation \eqref{eq:UfOU}.
Integration-by-parts, representation \eqref{eq:B}, and the stochastic Fubini theorem yield
\begin{eqnarray*}
U^{\theta,\sigma,H}_t 
&=& \sigma \e^{-\theta t}\int_0^t \e^{\theta s}\, \d B^H_s \\
&=& \sigma\e^{-\theta t}\left[\e^{\theta t} B_t^H -\theta \int_0^t B_s^H\, \e^{\theta s}\d s\right] \\
&=&
\int_0^t \sigma K_H(t,s)\, \d W_s -\theta  \sigma\e^{-\theta t}\int_0^t \int_0^s K_H(s,u)\, \d W_u\, \e^{\theta s}\d s  \\
&=&
\int_0^t \sigma K_H(t,s)\, \d W_s  -\theta  \sigma\e^{-\theta t}\int_0^t \int_s^t K_H(u,s)\, \e^{\theta u}\d u \, \d W_s \\
&=&
\int_0^t \left[\sigma K_H(t,s)-\theta\sigma\e^{-\theta t}\int_s^t K_H(u,s)\e^{\theta u}\, \d u\right]
\d W_s.
\end{eqnarray*}	
Thus \eqref{eq:UfOU} follows.

Let us then consider the representation \eqref{eq:WfOU}. By \eqref{eq:W}, \eqref{eq:Langevin}, and the stochastic Fubini theorem we have
\begin{eqnarray*}
W_t &=& \int_0^t K^{-1}_H(t,s)\, \d B^H_s \\
&=&
\frac{1}{\sigma}\int_0^t K^{-1}_H(t,s)\left[\d U^{\theta,\sigma,H}_s+\theta U^{\theta,\sigma,H}_t\d s \right] \\
&=&
\frac{1}{\sigma}\int_0^t K^{-1}_H(t,s)\d U^{\theta,\sigma,H}_s 
+\frac{\theta}{\sigma} \int_0^t \int_0^s \d U^{\theta,\sigma,H}_u\d s \\
&=&
\frac{1}{\sigma}\int_0^t K^{-1}_H(t,s)\d U^{\theta,\sigma,H}_s 
+\frac{\theta}{\sigma} \int_0^t \int_s^t \d u\d U^{\theta,\sigma,H}_s \\
&=&
\int_0^t \left[
\frac{1}{\sigma}K^{-1}_H(t,s) 
+\frac{\theta}{\sigma}(t-s)
\right]\d U^{\theta,\sigma,H}_s.
\end{eqnarray*}	
Representation \eqref{eq:WfOU} follows from this.
\end{proof}

Let $\Lambda_{\theta,\sigma,H}([0,T])$ be the integrand space of the fractional Ornstein--Uhlenbeck process over the time interval $[0,T]$.  As in the case of the fractional Brownian motion in Section \ref{sec:fbm-transfer}, we let $L^*_{\theta,\sigma,H}$ to be the linear operator from $\Lambda_{\theta,\sigma,H}([0,T])$ to $L^2([0,T])$ defined for simple functions as
$$
L^*_{\theta,\sigma,H}[g](t) = g(t)L_{\theta,\sigma,H}(T,t) + \int_t^T [g(s)-g(t)] L_{\theta,\sigma,H}(\d s, t).
$$
Then $L_{\theta,\sigma,H}\colon \Lambda_{\theta,\sigma,H}([0,T])\to L^2([0,T])$ is an isometry and its inverse is, for nice enough $g$, given by
$$
(L^*_{\theta,\sigma,H})^{-1}[g](t) = g(t)L^{-1}_{\theta,\sigma,H}(T,t) + \int_t^T [g(s)-g(t)] L^{-1}_{\theta,\sigma,H}(\d s, t).
$$
By plugging in $L_{\theta,\sigma,H}$, $L^{-1}_{\theta,\sigma,H}$, and using \eqref{eq:K-star}-\eqref{eq:K-star-inv} together with linearity yields explicit formulas
\begin{align}
\label{eq:L-star}
L^*_{\theta,\sigma,H}[g](t) &= c_H \sigma t^{\frac12-H}I_{T-}^{H-\frac12}[(\cdot)^{H-\frac12}g](t) - g(t)\theta \sigma e^{-\theta T}\int_t^T K_H(u,t)e^{\theta u}du \\
&+ \theta\sigma\int_t^T [g(s)-g(t)]\left[\theta e^{-\theta s}\int_t^s K(u,t)e^{\theta u}du- K_H(s,t)\right]ds
\end{align}
and 
\begin{equation}
\label{eq:L-star-inverse}
(L^*_{\theta,\sigma,H})^{-1}[g](t) = \frac{1}{c_H\sigma } t^{\frac12-H}I_{T-}^{\frac12-H}[(\cdot)^{H-\frac12}g](t) + \frac{\theta}{\sigma}(T-t)g(t) + \frac{\theta}{\sigma}\int_t^T [g(s)-g(t)]ds.
\end{equation}
These give relatively simple expressions for the operators $L^*_{\theta,\sigma,H}$ and $(L^*_{\theta,\sigma,H})^{-1}$ that can be used in the following extended transfer principle concerning abstract Wiener integration:
\begin{thm}[Extended transfer principle]\label{thm:etf}
For any $f\in L^2([0,T])$ and $g\in\Lambda_H([0,T])$ we have
\begin{eqnarray*}\label{eq:intWfOU}
	\int_0^T f(t)\, \d W_t &=& \int_0^T (L^*_{\theta,\sigma,H})^{-1}f(t)\, \d U^{\theta,\sigma,H}_t, \\
	\label{eq:intUfOU}
	\int_0^T g(t)\, \d U^{\theta,\sigma,H}_t &=& \int_0^T L^*_{\theta,\sigma,H} g(t)\, \d W_t.
\end{eqnarray*}
\end{thm}
\begin{proof}
The claim essentially follows from the same arguments as in the case of the fractional Brownian motion. Indeed, observe that the claim for simple indicator functions is nothing more than the statement of Theorem \ref{thm:tp}. The claim for arbitrary simple functions then follows by linearity, from which the proof is completed by the very definition of operators $L^*_{\theta,\sigma,H}$ and $(L^*_{\theta,\sigma,H})^{-1}$ as linear closures.
\end{proof}
Finally, we note that the extended transfer principle Theorem \ref{thm:etf} can be further extended to cover Skorokhod integration.  We refer to Al\`os et al. \cite{Alos-Mazet-Nualart-2001} for details.

\section{Application to prediction}
\label{sec:prediction}
With the same arguments as given in \cite{Sottinen-Viitasaari-2017b} we obtain the following prediction formula for fractional Ornstein--Uhlenbeck processes. As the arguments of \cite{Sottinen-Viitasaari-2017b} carry through (with obvious changes in the associated operators) once transfer principle is established, we leave the details to the reader.
\begin{pro}[Prediction]
Denote
\begin{equation*}\label{eq:Psi}
	\Psi_{\theta,\sigma,H}(t,s|u) =
	(L^*_{\theta,\sigma,H})^{-1}[L_{\theta,\sigma,H}(t,\cdot) - L_{\theta,\sigma,H}(u,\cdot)](s)
\end{equation*}
The conditional process $t\mapsto U^{\theta,\sigma,H}_t|\F_u^{U^{\theta,\sigma,H}}$, $t\ge u$, is Gaussian with mean
\begin{eqnarray*}
\hat m^{\theta,\sigma,H}_t(u) &=& \E\left[U^{\theta,\sigma,H}_t \big| \F^{\theta,\sigma,H}_u\right]\\
&=& U^{\theta,\sigma,H}_u - \int_0^u \Psi_{\theta,\sigma,H}(t,s|u)\, \d U^{\theta,\sigma,H}_s
\end{eqnarray*}
and covariance
\begin{eqnarray*}
\hat R_{\theta,\sigma,H}(t,s|u) &=& \mathbb{C}\mathrm{ov}\left[U^{\theta,\sigma,H}_t,U^{\theta,\sigma,H}_s\big|\F_u^{U^{\theta,\sigma,H}}\right] \\
&=&
R_{\theta,\sigma,H}(t,s) - \int_0^u L_{\theta,\sigma,H}(t,v) L_{\theta,\sigma,H}(s,v)\, \d v. 
\end{eqnarray*}
Here 
$$
R_{\theta,\sigma,H}(t,s) =
\mathbb{C}\mathrm{ov}\left[U^{\theta,\sigma,H}_t,U^{\theta,\sigma,H}_s\right].
$$
\end{pro}

\bibliographystyle{siam}
\bibliography{pipliateekki}
\end{document}